\documentclass[11pt, reqno]{amsart}
\usepackage[dvipsnames,usenames]{color}
\usepackage{amsmath}
\usepackage{amsfonts}
\usepackage{amssymb}
\usepackage{mathtools}
\usepackage{mathrsfs}
\usepackage{amsthm}
\usepackage{hyperref}
\usepackage{url}
\usepackage{comment}
\date{}
\newtheorem{theorem}{Theorem}
\newtheorem{thm}{Theorem}
\newtheorem{lemma}{Lemma}

\newtheorem{cor}{Corollary}
\newtheorem{prop}{Proposition}
\theoremstyle{definition}

\numberwithin{equation}{section}
\newtheorem{rem}{Remark}

\allowdisplaybreaks[4]
\newcommand{\abs}[1]{\left\vert#1\right\vert}

\newcommand{\D}{\mbox{$\mathbb{D}$}}
\theoremstyle{remark}

\numberwithin{equation}{section}
\newcommand{\ID}{{\mathbb D}}

\begin{document}
	\setcounter{page}{1}

	\title[Integral mean estimates for $(\alpha,\beta)$-harmonic functions]{Integral mean estimates for $(\alpha,\beta)$-harmonic functions}
	
	\author[Z.-G. Wang, Brindha Valson E and R. Vijayakumar]{Zhi-Gang Wang*, Brindha Valson E and R. Vijayakumar}

	\address{\noindent Zhi-Gang Wang\vskip.03in
		School of Mathematics and Statistics, Hunan First Normal University, Changsha 410205, Hunan, P. R. China.}	
	\email{sjyzhigangwang$@$hnfnu.edu.cn}
	
	\address{\noindent Brindha Valson E \vskip.03in
        Department of Mathematics,
		National Institute of Technology Calicut, Calicut-673 601, India.}
	\email{brindhavalson$@$gmail.com}

    \address{\noindent R. Vijayakumar\vskip.03in
        Department of Mathematics,
		National Institute of Technology Calicut, Calicut-673 601, India.}
	\email{mathesvijay8$@$gmail.com}

	

	
	\thanks{$^*$Corresponding author.}
	
	\date{\today}

	\begin{abstract}
We establish sharp $L^p$ integral mean estimates for $(\alpha,\beta)$-harmonic functions on the unit disk. Explicit bounds for the functions and their partial derivatives are obtained in terms of boundary data, by means of the associated Poisson-type kernel and hypergeometric function representations. As applications, we derive coefficient estimates and Hardy space-type results, extending well-known inequalities for classical harmonic and $\alpha$-harmonic functions to the $(\alpha,\beta)$-harmonic setting.
	\end{abstract} \maketitle

	
	\section{Introduction}
	Let $\mathbb{D} = \{z : |z| < 1\}$ be the open unit disk and $\mathbb{T} = \{z : |z| = 1\}$ the unit circle.
We denote by $C^m(\Omega)$ the space of all complex-valued $m$-times continuously differentiable
functions on $\Omega$, where $\Omega \subset \mathbb{C}$ and $m \in \mathbb{N} \cup \{0\}$.
In particular, $C(\Omega) := C^0(\Omega)$ stands for the space of all continuous functions on $\Omega$.
Let $\mathbb{Z}^-$ denote the set of negative integers.

The second-order partial differential operator in $\mathbb{D}$ is given by
\[
\Delta_{\alpha,\beta}
= \bigl(1-|z|^2\bigr)\frac{\partial^2}{\partial z\partial\overline{z}}
+ \alpha z\frac{\partial}{\partial z}
+ \beta \overline{z}\frac{\partial}{\partial\overline{z}}
- \alpha\beta,
\]
where $\alpha, \beta \in \mathbb{C}$, and
\[
\frac{\partial}{\partial z}
= \frac12\left(\frac{\partial}{\partial x} - i\frac{\partial}{\partial y}\right),\quad
\frac{\partial}{\partial\overline{z}}
= \frac12\left(\frac{\partial}{\partial x} + i\frac{\partial}{\partial y}\right).
\]
We are concerned with the homogeneous equation
\begin{equation}\label{Operator}
\Delta_{\alpha,\beta}\, w = 0
\end{equation}
in $\mathbb{D}$.
A twice continuously differentiable function $w$ in $\mathbb{D}$
is called $(\alpha,\beta)$-harmonic if it satisfies equation \eqref{Operator}.
If $w$ is $(\alpha,\beta)$-harmonic, then $\overline{w}$ is $(\beta,\alpha)$-harmonic.
$(\alpha,\beta)$-harmonic functions constitute a natural and important class of solutions to degenerate elliptic equations, unifying classical harmonic, $\alpha$-harmonic, and several other canonical function classes. They play a significant role in potential theory, harmonic analysis, and complex analysis, with close connections to Poisson kernels, integral means, Hardy spaces, and hypergeometric functions.

For $\alpha > -1$, $(0,\alpha)$-harmonic functions are simply referred to as $\alpha$-harmonic functions.
When $\alpha >-1$, $(\frac{\alpha}{2},\frac{\alpha}{2})$-harmonic functions
coincide with the real-kernel $\alpha$-harmonic functions.
In particular, $(0,0)$-harmonic functions are just the classical harmonic functions. 
This one-parameter family has been studied in \cite{cw,lr,Olofsson,OlofssonWittsten}.

We consider the associated Dirichlet boundary value problem
for functions $w$ satisfying equation \eqref{Operator}, namely,
\begin{equation}\label{Dirichlet_problem}
\begin{cases}
\Delta_{\alpha,\beta} w = 0 & \text{in } \mathbb{D}, \\
w = f & \text{on } \mathbb{T},
\end{cases}
\end{equation}
where the boundary function $f \in \mathcal{C}(\mathbb{T})$,
and the boundary condition in \eqref{Dirichlet_problem}
is understood as $w_r \to f$ in $\mathcal{C}(\mathbb{T})$ as $r \to 1^-$,
with $w_r\left(e^{i\theta}\right)=w\left(re^{i\theta}\right)$ for 
$e^{i\theta}\in\mathbb{T}$ and $r\in[0,1)$.

An important example of $(\alpha,\beta)$-harmonic functions
is the function $u_{\alpha,\beta}$ with parameters $\alpha,\beta\in\mathbb{C}$,
defined by
\begin{equation}\label{u_alpha_beta}
u_{\alpha,\beta}(z)
=
\frac{\left(1-|z|^2\right)^{\alpha+\beta+1}}
{(1-z)^{\alpha+1}(1-\overline{z})^{\beta+1}}
\quad (z\in\mathbb{D}).
\end{equation}
This function plays a key role in the theory of $(\alpha,\beta)$-harmonic functions,
see e.g., \cite{Schwarz_for_pq,ag,km,KlintbergOlofsson,q}.

This paper extends the $H^p$ theory for $M$-harmonic functions in \cite{H^p_for_M_harmonic} to the setting of bi-parameter $(\alpha,\beta)$-harmonic functions. 
By using refined estimates for the Poisson-type kernel, we establish more general sharp integral inequalities. 
In addition, we study integral mean estimates for $(\alpha,\beta)$-harmonic functions on the unit disk $\D$. 
We obtain sharp $L^p$ integral mean estimates for such functions, as well as explicit bounds for the functions and their partial derivatives in terms of boundary data, via the corresponding Poisson-type kernels and hypergeometric function representations. 
As applications, we derive coefficient estimates and Hardy space-type results, which extend several classical inequalities for harmonic and $\alpha$-harmonic functions to the general $(\alpha,\beta)$-harmonic framework.

	\vskip .05in
	 
	\section{Preliminaries}
We begin this section by recalling several basic definitions and relevant results for our study.
\subsection{Jensen's inequality}
	
	Suppose \(\varphi : [c, d] \to \mathbb{R}\) is convex, and the functions \(f\), \(p\) are integrable on \([a,b]\), where \(c < d\) and \(a < b\). If for any \(x \in [a,b]\), \(f(x) \in [c, d]\), \(p(x) \ge 0\), and \(\int_a^b p(x) dx > 0\), then
	\[
	\varphi\left(\frac{\int_a^b f(x) p(x) dx}{\int_a^b p(x) dx}\right) \le \frac{\int_a^b \varphi(f(x)) p(x) dx}{\int_a^b p(x) dx}.
	\]
	
	\subsection{Hardy space}
	For $p \in (0, \infty]$ and a measurable complex-valued function $f$ defined on $\mathbb{D}$, the integral mean of $f$ is defined by
	$$M_p(r,f) = \left(\frac{1}{2\pi} \int_0^{2\pi} \abs{f\left(re^{i\theta}\right)}^p d\theta \right)^{1/p}\quad (0<p<\infty)$$
	and
    $$M_\infty(r,f) = \underset{0\le\theta\le2\pi}{\mathrm{ess}\sup}\left|f\left(re^{i\theta}\right)\right|.$$

	The generalized Hardy space $H_G^p(\ID)$ consists of all measurable functions $f$ from $\ID$ to $\mathbb{C}$ such that $M_p(r,f)$ bounded as $r\rightarrow 1$.  The classical Hardy space $H^p(\ID)$ (resp. $h^p(\ID)$) is the set of all elements of $H_G^p(\ID)$  which are analytic (resp. harmonic) on $\ID$ (see  \cite{Duren,w}).
    
	We denote by $h^p_{\alpha,\beta}(\ID)$ the corresponding Hardy space for $(\alpha,\beta)$-harmonic functions. A comprehensive study on $h^p_{\alpha,\beta}(\ID)$ can be found in \cite{H^p_for_M_harmonic}. 
	\subsection{Lebesgue measurable space and subharmonic functions} Denote by $L^p(\mathbb{T})$, for $p\in[1,\infty]$, the space of all measurable functions $f:\mathbb{T}\rightarrow\mathbb{C}$ such that
	\begin{equation*}
		\|f\|_{L^p}=
		\begin{cases}
			\left(\frac{1}{2\pi} \int_0^{2\pi} |f(e^{i\theta})|^p d\theta \right)^{1/p}   &(1\leq p < \infty), \\
			{\mathrm{ess}\sup}\left|f\left(e^{i\theta}\right)\right| & (p=\infty).
		\end{cases}
	\end{equation*}
Moreover, let $G$ be an open connected subset of $\mathbb{C}$, and let $f:G\rightarrow\mathbb{R}$ be continuous. We say that $f$ is subharmonic if for every closed disk $\overline{B(a,r)}\subset G$,
  $$f(a)\leq \frac{1}{2\pi}\int_0^{2\pi}f\left(a+re^{i\theta}\right)d\theta.$$
Recall that if $\Delta f\geq 0$ in $G$, then $f$ is subharmonic in $G$.

	\subsection{Gauss hypergeometric functions}
	
	The Gauss hypergeometric function is defined by the series
	$$
	F(a, b; c; z) = \sum_{n=0}^{\infty} \frac{(a)_n (b)_n}{(c)_n} \frac{z^n}{n!}
	$$
	for $a,b,c \in \mathbb{C}$ such that $c\neq-1,-2,\dots,$ where \((a)_0 = 1\) and $$(a)_n = a(a+1) \cdots (a+n-1)=\frac{\Gamma(a+n)}{\Gamma(a)} \quad (n=1,2,\ldots)$$ is the Pochhammer or ascending factorial notation. The standard Gamma function is defined by $$\Gamma(z)=\int_{0}^{\infty}t^{z-1}e^{-t}\,dt$$ for $\mbox{Re}(z)>0$. It is well-known that $\Gamma$ continues to a meromorphic function in $\mathbb{C}$ with simple poles at the points $z=0,-1,-2,\dots$.
	
	We now recall the well-known results \cite[Chapter 2] {Specialfunctions} about Gaussian hypergeometric functions:
	\begin{equation}\label{F(a,b;c;x)prop_1}
		\lim_{x \to 1} F(a,b;c;x) = \frac{\Gamma(c) \Gamma(c - a - b)}{\Gamma(c - a)\Gamma(c - b)}\quad  (\mbox{Re}(c - a - b) > 0)
	\end{equation}
	 and 
     \begin{equation}\label{F(a,b;c;x)prop_2}
     \frac{dF(a,b;c;x)}{dx} = \frac{ab}{c} F(a+1, b+1; c+1; x).
     \end{equation}

      \begin{lemma} {\rm(See \cite{KlintbergOlofsson})}\label{Poisson_integral_Rep}
		Let $\alpha,\beta \in \mathbb{C} \setminus \mathbb{Z}^-$ such that $\mbox{\rm Re}(\alpha+\beta)>-1.$ Let $f\in \mathcal{C}(\mathbb{T})$. Then a function $w$ in $\ID$ satisfies \eqref{Dirichlet_problem} if and only if it has the form of a Poisson type integral
		\begin{equation}\label{poisson_integral}
			w(z) = K_{\alpha,\beta}[f](z) = \frac{1}{2\pi} \int_0^{2\pi} K_{\alpha,\beta}\left(z e^{-it}\right) f\left(e^{it}\right)dt,
		\end{equation}
		where
		\[
		K_{\alpha,\beta}(z) = c_{\alpha,\beta} \frac{(1 - |z|^2)^{\alpha + \beta + 1}}{(1 - z)^{\alpha+1} (1 - \overline{z})^{\beta+1}}
		\]
		is the $(\alpha,\beta)$-harmonic Poisson kernel in $\ID$, and
		\begin{equation}\label{23}
		c_{\alpha,\beta} = \frac{\Gamma(\alpha+1) \Gamma(\beta+1)}{\Gamma(\alpha + \beta + 1)}.
		\end{equation}
	\end{lemma}

     \begin{lemma}{\rm(See \cite{Specialfunctions})}\label{Lemma1}
		Let \(c > 0,\, a \le c,\, b \le c\) and \(ab \le 0\) (or \(ab \ge 0\)). Then the function \(F(a,b;c;x)\) is decreasing (or increasing) on \((0,1)\).
	\end{lemma}

    \begin{lemma}{\rm(See \cite{km})} Let $u_{\alpha,\beta}(z)$ be the $(\alpha,\beta)$-harmonic function  defined in \eqref{u_alpha_beta}. Then
\label{partial derivative of K}
\begin{align*}
\begin{split}
\partial_z u_{\alpha,\beta}(z)=& \left(\frac{(\alpha+1)(1-\overline{z})}{(1-z)\left(1-|z|^2\right)}-\frac{\beta\overline{z}}{1-|z|^2}\right)u_{\alpha,\beta}(z) \\=& \frac{(\alpha + 1)\left(1-|z|^2\right)^{\alpha+\beta}}{(1-z)^{\alpha+2}(1-\overline{z})^\beta} - \frac{\beta \overline{z} \left(1-|z|^2\right)^{\alpha+\beta}}{(1-z)^{\alpha+1}(1-\overline{z})^{\beta+1}}
\end{split}
\end{align*}
and
\begin{align*}
\begin{split}
\partial_{\overline{z}} u_{\alpha,\beta}(z)=& \left(\frac{(\beta+1)(1-z)}{(1-\overline{z})(1-|z|^2)}-\frac{\alpha z}{1-|z|^2}\right)u_{\alpha,\beta}(z)\\=& \frac{(\beta + 1)\left(1-|z|^2\right)^{\alpha+\beta}}{(1-z)^{\alpha}(1-\overline{z})^{\beta+2}} - \frac{\alpha z \left(1-|z|^2\right)^{\alpha+\beta}}{(1-z)^{\alpha+1}(1-\overline{z})^{\beta+1}}.
\end{split}
\end{align*}
\end{lemma}

\begin{lemma}{\rm (See \cite{Specialfunctions})}\label{lem4}
	It holds that $$\int_0^{\pi}\frac{dt}{(1+r^2-2r\cos t)^v}=\pi F\left(v,v;1;r^2\right),$$ where $F(a,b;c;x)$ is the Gauss hypergeometric function.
	\end{lemma}

    \begin{lemma}{\rm(See \cite{Schwarz_for_pq})}{\label{Bound_of_u_alpha_beta}}
	Let $u_{\alpha,\beta}$ be as in \eqref{u_alpha_beta} for some $\alpha,\beta \in \mathbb{C}$ and $k,l$ are non-negative integers. Then
	\[
	\left| \partial^{k} \bar{\partial}^{\,l} u_{\alpha,\beta}(z) \right|
	\leq
	C_{\alpha,\beta,k,l}\,
	\frac{|u_{\alpha,\beta}(z)|}{(1-|z|^2)^{k+l}},
	\]
	where $C_{\alpha,\beta,k,l}$ is a  constant that depends on $\alpha,\beta,k\text{ and }l.$
	\end{lemma}

	\begin{lemma} {\rm (See \cite{samyVurinen,YangChu})}\label{seriesconverge}
	Let $r_n$ and $s_n$ $(n=0,1,2,\ldots)$ be real numbers, and let the power series
		\[
		R(x)=\sum_{n=0}^{\infty} r_n x^n
		\ \ \text{and} \ \
		S(x)=\sum_{n=0}^{\infty} s_n x^n
		\]
		be convergent for $|x|<r\, (r>0)$ with $s_n>0$ for all $n$.
		If the non-constant sequence $\{r_n/s_n\}$ is increasing (decreasing) for all $n$,
		then the function $x \mapsto R(x)/S(x)$ is strictly increasing (resp.\ decreasing)
		on $(0,r)$.
	\end{lemma}

	\vskip .05in
	\section{Integral means of $(\alpha,\beta)$-harmonic functions}

    Our first main theorem gives the sharp $L^p$ integral mean bound for $(\alpha,\beta)$-harmonic functions solving the Dirichlet problem \eqref{Dirichlet_problem}, which is the foundational estimate for all subsequent results in this section.
	
	\begin{theorem}\label{Integral mean thm}
		Suppose that $\alpha, \beta\in \mathbb{C}\setminus\mathbb{Z}^-$ and $\mbox{\rm Re}(\alpha+\beta) > -1$. Let $w(z)=K_{\alpha,\beta}[f](z)$ be an $(\alpha,\beta)$-harmonic function defined on  $\mathbb{D}$ and satisfying condition \eqref{Dirichlet_problem} with $f\in L^p(\mathbb{T}),\,1\leq p<\infty$. Then for $z=re^{i\theta}\in \mathbb{D}$, the following sharp inequality holds:
		\begin{align}\begin{split}\label{integral_mean_bound}
			 M_p(r,w)\leq& |c_{\alpha,\beta}|\exp\left(\frac{\pi}{2}| \text{\rm Im}(\alpha-\beta)|\right)\\&\quad \times F\left(-\frac{\mbox{\rm Re}(\alpha+\beta)}{2},-\frac{\mbox{\rm Re}(\alpha+\beta)}{2};1;r^2\right)\|f\|_{L^p(\mathbb{T})},
            \end{split}
		\end{align} where $c_{\alpha,\beta}$ is given by \eqref{23}.
		 Moreover, $$M_p(r,w)\leq \exp\left(\frac{\pi}{2}|\text{\rm Im}(\alpha-\beta)|\right)B(\alpha,\beta)\|f\|_{L^p(\mathbb{T})},$$ where  $$B(\alpha,\beta):=\frac{\Gamma(\mbox{\rm Re}(\alpha+\beta)+1)}{\Gamma^2\left(\frac{\mbox{\rm Re}(\alpha+\beta)}{2}+1\right)}.$$
	\end{theorem}
	\begin{proof}
		Let  $w(z)=K_{\alpha,\beta}[f](z)$ be an $(\alpha,\beta)$-harmonic function defined on  $\mathbb{D}$.
		For $z=re^{i\theta}$, it follows from \eqref{poisson_integral} in Lemma \ref{Poisson_integral_Rep} that
		\begin{equation}\label{Thm1eq1}
			|w(z)|\leq  \frac{1}{2\pi}\int_0^{2\pi}\abs{c_{\alpha,\beta}}\abs{\frac{\left(1 - |z|^2\right)^{\alpha + \beta + 1}}{(1 - ze^{-it})^{\alpha+1} (1 - \overline{z}e^{it})^{\beta+1}}}\abs{f\left(e^{it}\right)}dt .
		\end{equation}
		By \cite[Theorem 6.4]{KlintbergOlofsson}, we obtain  \begin{equation}\label{WVB_3.3}
        \bigg|\frac{\left(1 - |z|^2\right)^{\alpha + \beta + 1}}{(1 - z)^{\alpha+1} (1 - \overline{z})^{\beta+1}}\bigg|\leq \exp\left(\frac{\pi}{2}| \text{\rm Im}(\alpha-\beta)|\right)\frac{\left(1-|z|^2\right)^{\text{\rm Re}(\alpha+\beta)+1}}{|1-z|^{\text{\rm Re}(\alpha+\beta)+2}}.
        \end{equation}
		Suppose that $$I=\frac{1}{2\pi}\int_0^{2\pi}|c_{\alpha,\beta}|\frac{(1-r^2)^{\text{\rm Re}(\alpha+\beta)+1}}{|1-re^{i(\theta-t)}|^{\text{\rm Re}(\alpha+\beta)+2}} d\theta.$$
		By \cite[Theorem 2.1]{PLiSamyLuo}, we get
		$$I=|c_{\alpha,\beta}|F\left(-\frac{\mbox{Re}(\alpha+\beta)}{2},-\frac{\mbox{Re}(\alpha+\beta)}{2};1;r^2\right).$$  Moreover, by Lemma \ref{Lemma1} and \eqref{F(a,b;c;x)prop_1}, we have
		\begin{align*}
        \begin{split}
		|c_{\alpha,\beta}|&F\bigg(-\frac{\mbox{Re}(\alpha+\beta)}{2},-\frac{\mbox{Re}(\alpha+\beta)}{2};1;r^2\bigg)\\\leq&	\lim_{r \to 1}|c_{\alpha,\beta}|F\left(-\frac{\mbox{Re}(\alpha+\beta)}{2},-\frac{\mbox{Re}(\alpha+\beta)}{2};1;r^2\right)\\=&|c_{\alpha,\beta}|\frac{\Gamma\left(\mbox{Re}(\alpha+\beta)+1\right)}{\Gamma^2\left(\frac{\mbox{Re}(\alpha+\beta)}{2}+1\right)}.
        \end{split}
		\end{align*}
		
		\noindent For $1\leq p<\infty$, by Jensen's inequality, we obtain
		\begin{align*}
        \begin{split}
			&|w(z)|^p\\\leq&\left(\exp\left(\frac{\pi}{2}| \text{Im}(\alpha-\beta)|\right)\right)^p\left(\frac{I}{2\pi} \int_{0}^{2\pi}\frac{|c_{\alpha,\beta}|\frac{(1-r^2)^{\text{Re}(\alpha+\beta)+1}}{|1-re^{i(\theta-t)}|^{\text{Re}(\alpha+\beta)+2}}}{I}|f(e^{it})| dt\right)^p\\\leq&\left(\exp\left(\frac{\pi}{2}| \text{Im}(\alpha-\beta)|\right)\right)^p \frac{I^{p-1}}{2\pi}\int_{0}^{2\pi}|c_{\alpha,\beta}|\frac{(1-r^2)^{\text{Re}(\alpha+\beta)+1}}{|1-re^{i(\theta-t)}|^{\text{Re}(\alpha+\beta)+2}}|f(e^{it})|^p dt.
     \end{split}
		\end{align*}
Integrating on both sides of the above inequality and using Fubini's theorem, we get
		\begin{align*}
        \begin{split}
			&\frac{1}{2\pi}\int_0^{2\pi}|w(e^{i\theta})|^p d\theta\\\leq & \left(\exp\left(\frac{\pi}{2}| \text{Im}(\alpha-\beta)|\right)\right)^p\\&\quad\times\frac{I^{p-1}}{2\pi}\int_{0}^{2\pi}\left(\frac{1}{2\pi}\int_{0}^{2\pi}|c_{\alpha,\beta}|\frac{(1-r^2)^{\text{Re}(\alpha+\beta)+1}}{|1-re^{i(\theta-t)}|^{\text{Re}(\alpha+\beta)+2}}|f(e^{it})|^p dt \right)d\theta\\\leq& \left(\exp\left(\frac{\pi}{2}| \text{Im}(\alpha-\beta)|\right)\right)^pI^p\|f\|^p_{L^p(\mathbb{T})}.
            \end{split}
		\end{align*}
		Thus, we find that
		\begin{align*}
        \begin{split}
			 M_p(r,w)\leq& \exp\left(\frac{\pi}{2}|\text{Im}(\alpha-\beta)|\right)I\|f\|_{L^p(\mathbb{T})}\\=&\exp\left(\frac{\pi}{2}| \text{Im}(\alpha-\beta)|\right)|c_{\alpha,\beta}|\\&\quad\times F\left(-\frac{\mbox{Re}(\alpha+\beta)}{2},-\frac{\mbox{Re}(\alpha+\beta)}{2};1;r^2\right)\|f\|_{L^p(\mathbb{T})}\\\leq& \exp\left(\frac{\pi}{2}| \text{Im}(\alpha-\beta)|\right)B(\alpha,\beta)\|f\|_{L^p(\mathbb{T})}.     
        \end{split}
		\end{align*}
        
		To prove the sharpness, suppose that $\alpha,\beta\in\mathbb{C}\setminus\mathbb{Z}^-$ satisfy $\mbox{Re}(\alpha+\beta)>-1$ and $\mbox{Im}(\alpha)=\mbox{Im}(\beta)$, $f:\mathbb{T}\rightarrow\mathbb{C}$ is defined by $f(e^{it})\equiv 1$. Then $$w_0(z)= c_{\alpha,\beta}F\left(-\frac{\mbox{Re}(\alpha+\beta)}{2},-\frac{\mbox{Re}(\alpha+\beta)}{2};1;r^2\right),$$ which shows that $w_0(z)$ is $(\alpha,\beta)$-harmonic.
		By observing that
		\begin{align*}
        \begin{split}
		&M_p(r,w_0)\\=&\bigg(\frac{1}{2\pi} \int_0^{2\pi}|w_0(z)|^p d\theta\bigg)^\frac{1}{p}\\=&\bigg(\frac{1}{2\pi} \int_0^{2\pi}\left(|c_{\alpha,\beta}|F\left(-\frac{\mbox{Re}(\alpha+\beta)}{2},-\frac{\mbox{Re}(\alpha+\beta)}{2};1;r^2\right)^p d\theta\right)^\frac{1}{p}\\=&|c_{\alpha,\beta}|F\left(-\frac{\mbox{Re}(\alpha+\beta)}{2},-\frac{\mbox{Re}(\alpha+\beta)}{2};1;r^2\right)\|f\|_{L^p(\mathbb{T})},
        \end{split}
		\end{align*}
		 we deduce that \eqref{integral_mean_bound} is sharp.
	\end{proof}
	As an immediate consequence of Theorem \ref{Integral mean thm}, we get the integral mean of real kernel $\alpha$-harmonic functions due to Long \cite[Theorem 1.1]{optimal_inequalities_alpha_harmonic}.
	\begin{cor}
		Let $w(z)$ be a real kernel $\alpha$-harmonic function with boundary function $f\in L^p(\mathbb{T}) , 1\leq p< \infty$, and $z=re^{i\theta}\in \ID$. Then the sharp inequality 
		$$M_p(r,w)\leq c_{\alpha}F\left(-\frac{\alpha}{2},-\frac{\alpha}{2};1;r^2\right)\|f\|_{L^p(\mathbb{T})}$$
	holds. 
	\end{cor}

    Now, we estimate the modulus of the first-order partial derivatives of $(\alpha,\beta)$-harmonic functions $w(z)$ in terms of the $L^p$-norm of the boundary function $f$.
	\begin{theorem}\label{bound_of_w_z}
		Let $w(z)=K_{\alpha,\beta}[f](z)$ be an $(\alpha,\beta)$-harmonic function with {$\text{\rm Re}(\alpha+\beta)>-1$} defined on  $\mathbb{D}$  with $f\in L^p(\mathbb{T}),1\leq p< \infty$ and $z=re^{i\theta}\in \ID$. Then
		\begin{itemize}
			\item [(i)]there exists a function $C_{\alpha,\beta,p}(r)$ and a constant $C_{\alpha,\beta,p}$ such that 
			\begin{align*}
            \begin{split}
				\abs{\frac{\partial}{\partial z}w(z)}&\leq \frac{C_{\alpha,\beta,p}(r)}{(1-r^2)^{1+1/p}}\exp\left(\frac{\pi}{2}|\text{\rm Im}(\alpha-\beta)|\right)\|f\|_{L^p(\mathbb{T})}\\&\leq \frac{C_{\alpha,\beta,p}}{(1-r^2)^{1+1/p}}\exp\left(\frac{\pi}{2}|\text{\rm Im}(\alpha-\beta)|\right)\|f\|_{L^p(\mathbb{T})},
                \end{split}
			\end{align*}
			where
			\begin{equation*}
				C_{\alpha,\beta,p}(r)=|c_{\alpha,\beta}|(|\alpha +1 |+|\beta| r)F_1
			\end{equation*}
            with $$F_1=\left(F\left(1-\frac{(\text{\rm Re}(\alpha+\beta)+2)q}{2},
            1-\frac{(\text{\rm Re}(\alpha+\beta)+2)q}{2};1;r^2\right)\right)^{1/q},$$
            and 
			\begin{equation*}
				C_{\alpha,\beta,p}=|c_{\alpha,\beta}|(|\alpha +1 |+|\beta|)\left(\frac{\Gamma((\text{\rm Re}(\alpha+\beta)+2)q-1)}{\Gamma^2\left(\dfrac{(\text{\rm Re}(\alpha+\beta)+2)q}{2}\right)}\right)^{1/q}.
			\end{equation*}
			The constant $C_{\alpha,\beta,p}$ is asymptotically sharp as $\alpha,\beta \rightarrow 0$.
			
			\item[(ii)] there exists a function $D_{\alpha,\beta,p}(r)$ and a constant $D_{\alpha,\beta,p}$ such that 
			\begin{align*}
            \begin{split}
				\bigg|\frac{\partial}{\partial \overline{z}}w(z)\bigg|&\leq \frac{D_{\alpha,\beta,p}(r)}{(1-r^2)^{1+1/p}}\exp\left(\frac{\pi}{2}|\text{\rm Im}(\alpha-\beta)|\right)\|f\|_{L^p(\mathbb{T})}\\&\leq \frac{D_{\alpha,\beta,p}}{(1-r^2)^{1+1/p}}\exp\left(\frac{\pi}{2}|\text{\rm Im}(\alpha-\beta)|\right)\|f\|_{L^p(\mathbb{T})},
                 \end{split}
			\end{align*}
			where $$D_{\alpha,\beta,p}(r)=|c_{\alpha,\beta}|(|\beta +1 |+|\alpha| r)F_1$$ and $$D_{\alpha,\beta,p}=|c_{\alpha,\beta}|(|\beta +1 |+|\alpha|)\left(\frac{\Gamma((\text{\rm Re}(\alpha+\beta)+2)q-1)}{\Gamma^2\left(\dfrac{(\text{\rm Re}(\alpha+\beta)+2)q}{2}\right)}\right)^{1/q}.$$
			The constant $D_{\alpha,\beta,p}$ is asymptotically sharp as $\alpha,\beta \rightarrow 0$.
		\end{itemize}
	\end{theorem}
	\begin{proof}
		Suppose that $w$ is an $(\alpha,\beta)$-harmonic function on $\mathbb{D}$  with $f\in L^p(\mathbb{T}),\, 1\leq p< \infty$. 
		Differentiating on both sides of \eqref{poisson_integral} with respect to $z$, we obtain 
					$$\frac{\partial}{\partial z}w(z)=\frac{1}{2\pi} \int_0^{2\pi} \frac{\partial}{\partial z}K_{\alpha,\beta}\left(z e^{-it}\right) f\left(e^{it}\right) dt.$$
		It follows that 
		\begin{equation}\label{3.3}
		\bigg|\frac{\partial}{\partial z}w(z)\bigg|\leq \frac{1}{2\pi} \int_0^{2\pi}\bigg|\frac{\partial}{\partial z}K_{\alpha,\beta}\left(z e^{-it}\right)\bigg|\abs{f\left(e^{it}\right)}dt.
		\end{equation}
		By Lemma \ref{partial derivative of K} and chain rule, we have
		\begin{align*}
        \begin{split}
			&\bigg|\frac{\partial}{\partial z}K_{\alpha,\beta}\left(z e^{-it}\right)\bigg|\\=&|c_{\alpha,\beta}|\bigg|\frac{(\alpha+1)(1-\overline{z}e^{it})}{(1-ze^{-it})(1-|z|^2)}-\frac{\beta\overline{z}e^{it}}{1-|z|^2}\bigg||u_{\alpha,\beta}|\\\leq&|c_{\alpha,\beta}|\exp\left(\frac{\pi}{2}|\text{\rm Im}(\alpha-\beta)|\right)\frac{|\alpha+1|+|\beta||\overline{z}|}{1-|z|^2}\frac{(1-|z|^2)^{\text{\rm Re}(\alpha+\beta)+1}}{|1-re^{i(\theta-t)}|^{\text{\rm Re}(\alpha+\beta)+2}}\\=& |c_{\alpha,\beta}|\exp\left(\frac{\pi}{2}|\text{\rm Im}(\alpha-\beta)|\right)(|\alpha+1|+|\beta||\overline{z}|)\frac{(1-|z|^2)^{\text{Re}(\alpha+\beta)}}{|1-re^{i(\theta-t)}|^{\text{\rm Re}(\alpha+\beta)+2}}.
            \end{split}
		\end{align*} 
		It follows from \eqref{3.3} that
	\begin{align*}
    \begin{split}
		&\bigg|\frac{\partial}{\partial z}w(z)\bigg|\\\leq& \frac{1}{2\pi}\exp\left(\frac{\pi}{2}|\text{\rm Im}(\alpha-\beta)|\right) \int_0^{2\pi}|c_{\alpha,\beta}|(|\alpha +1 |+|\beta| |\overline{z}|)\frac{(1-|z|^2)^{\text{Re}(\alpha+\beta)}}{|1-ze^{-it}|^{\text{Re}(\alpha+\beta)+2}}\abs{f\left(e^{it}\right)}dt\\=& \frac{|c_{\alpha,\beta}|}{2\pi}(|\alpha +1 |+|\beta| |\overline{z}|)\exp\left(\frac{\pi}{2}|\text{\rm Im}(\alpha-\beta)|\right) \int_0^{2\pi}\frac{(1-|z|^2)^{\text{Re}(\alpha+\beta)}}{|1-ze^{-it}|^{\text{Re}(\alpha+\beta)+2}}\abs{f\left(e^{it}\right)} dt\\\leq& |c_{\alpha,\beta}|(|\alpha +1 |+|\beta| |\overline{z}|)\exp\left(\frac{\pi}{2}|\text{\rm Im}(\alpha-\beta)|\right)\left(\frac{1}{2\pi}\int_0^{2\pi}\frac{(1-|z|^2)^{(\text{Re}(\alpha+\beta))q}}{|1-ze^{-it}|^{(\text{Re}(\alpha+\beta)+2)q}}dt\right)^{1/q}\\&\times\left(\frac{1}{2\pi}\int_0^{2\pi}\abs{f\left(e^{it}\right)}^pdt\right)^{1/p}.
        \end{split}
	\end{align*}
	The last inequality is due to H\"older's inequality for $\frac{1}{p}+\frac{1}{q}=1$	and $1<p\leq \infty$.
    
		For $z=re^{i\theta}$, we denote
		\begin{equation}\label{I1}
			I_1:=\int_0^{2\pi}\frac{(1-r^2)^{(\text{Re}(\alpha+\beta))q}}{|1-re^{i(\theta-t)}|^{(\text{Re}(\alpha+\beta)+2)q}}dt.
		\end{equation} 
		By the following variable changes $$e^{i(t-\theta)}=\frac{r-e^{is}}{1-re^{is}},$$ $$\abs{1-re^{i(\theta-t)}}=\frac{1-r^2}{\abs{1-re^{-is}}}$$ and $$dt=\frac{1-r^2}{|1-re^{is}|^2}ds,$$ combining \eqref{F(a,b;c;x)prop_1} with lemmas \ref{Lemma1} and \ref{lem4}, we find from \eqref{I1} that
		\begin{align*}
           \begin{split}
			I_1=&\int_0^{2\pi}\frac{(1-r^2)^{(\text{Re}(\alpha+\beta))q}}{\left(\dfrac{1-r^2}{|1-re^{-is}|}\right)^{(\text{Re}(\alpha+\beta)+2)q}}\frac{1-r^2}{|1-re^{is}|^2}ds\\=& \int_0^{2\pi}(1-r^2)^{1-2q}|1-re^{is}|^{(\text{Re}(\alpha+\beta)+2)q-2} ds\\=&(1-r^2)^{1-2q}\int_0^{2\pi}\big(|1-re^{is}|^2\big)^{\frac{(\text{Re}(\alpha+\beta)+2)q}{2}-1} ds\\=&(1-r^2)^{1-2q}\int_0^{2\pi}(1+r^2-2r\cos s)^{\frac{(\text{Re}(\alpha+\beta)+2)q}{2}-1} ds\\=&2(1-r^2)^{1-2q}\int_0^{\pi}(1+r^2-2r\cos s)^{\frac{(\text{Re}(\alpha+\beta)+2)q}{2}-1} ds\\=&2\pi(1-r^2)^{1-2q}F\bigg(1-\frac{(\text{Re}(\alpha+\beta)+2)q}{2},1-\frac{(\text{Re}(\alpha+\beta)+2)q}{2};1;r^2\bigg)\\\leq& 2\pi(1-r^2)^{1-2q} \frac{\Gamma((\text{Re}(\alpha+\beta)+2)q-1)}{\Gamma^2\left(\dfrac{(\text{Re}(\alpha+\beta)+2)q}{2}\right)}.   
            \end{split}
		\end{align*}
		Let $$C_{\alpha,\beta,p}(r)=|c_{\alpha,\beta}|(|\alpha +1 |+|\beta| r)F_1.$$
		Then 
		\begin{align*}
        \begin{split}
			\bigg|\frac{\partial}{\partial z}w(z)\bigg|\leq& |c_{\alpha,\beta}|(|\alpha +1 |+|\beta| r)\exp\left(\frac{\pi}{2}|\text{\rm Im}(\alpha-\beta)|\right)I_1^{1/q}\left(\int_0^{2\pi}\abs{f\left(e^{it}\right)}^pdt\right)^{1/p}\\\leq&\frac{C_{\alpha,\beta,p}(r)}{(1-r^2)^{1+1/p}}\exp\left(\frac{\pi}{2}|\text{\rm Im}(\alpha-\beta)|\right)\|f\|_{L^p(\mathbb{T})}.
            \end{split}
		\end{align*}
		Furthermore, we have $$\bigg|\frac{\partial}{\partial z}w(z)\bigg|\leq\frac{C_{\alpha,\beta,p}}{(1-r^2)^{1+1/p}}\exp\left(\frac{\pi}{2}|\text{\rm Im}(\alpha-\beta)|\right)\|f\|_{L^p(\mathbb{T})},$$ where $$C_{\alpha,\beta,p}=|c_{\alpha,\beta}|(|\alpha +1 |+|\beta| )\left(\frac{\Gamma((\text{Re}(\alpha+\beta)+2)q-1)}{\Gamma^2\bigg(\dfrac{(\text{Re}(\alpha+\beta)+2)q}{2}\bigg)}\right)^{1/q}.$$ If $\alpha=\beta=0$, then $$C_{\alpha,\beta,p}(r)=\left(F\left(1-q,1-q;1;r^2\right)\right)^{1/q}$$ and $$C_{0,0,p}=\bigg(\frac{\Gamma(2q-1)}{\Gamma^2(q)}\bigg)^{1/q}.$$
        
 		Now we will show that the constant $C_{0,0,p}$ is sharp. For $0<\rho<1$, consider the function 
		\begin{equation*}
			f_\rho\left(e^{is}\right)=\left(1-\rho^2\right)^{\frac{2}{1-p}}(1+\cos s)^{\frac{1}{p-1}}e^{is}. 
		\end{equation*} 
		Take $$w(z)=K_{\alpha,\beta}[f_\rho](z).$$
		By using the similar method as in \cite{optimal_inequalities_alpha_harmonic}, we find that $$\lim\limits_{r=\rho\rightarrow1}\frac{(1-r^2)^{1+1/p}|w_z|}{\|f\|_{L^p(\mathbb{T})}}=C_{0,0,p}.$$
		This means that the constant $C_{0,0,p}$ is sharp.
        
        Similarly, we have
		$$\frac{\partial}{\partial \overline{z}}w(z)=\frac{1}{2\pi} \int_0^{2\pi} \frac{\partial}{\partial \overline{z}}K_{\alpha,\beta}\left(z e^{-it}\right) f\left(e^{it}\right) dt,$$
		it follows that 
		\begin{equation*}
			\bigg|\frac{\partial}{\partial \overline{z}}w(z)\bigg|\leq \frac{1}{2\pi} \int_0^{2\pi}\bigg|\frac{\partial}{\partial \overline{z}}K_{\alpha,\beta}\left(z e^{-it}\right)\bigg|\abs{f\left(e^{it}\right)}dt.
		\end{equation*}
		Moreover, by Lemma \ref{partial derivative of K} and chain rule, we obtain
		\begin{align*}
        \begin{split}
			&\bigg|\frac{\partial}{\partial \overline{z}}K_{\alpha,\beta}\left(z e^{-it}\right)\bigg|\\ \leq &|c_{\alpha,\beta}|(|\beta +1 |+|\alpha| |z|)\exp\left(\frac{\pi}{2}|\text{\rm Im}(\alpha-\beta)|\right)\frac{(1-|z|^2)^{\text{Re}(\alpha+\beta)}}{|1-ze^{-it}|^{\text{Re}(\alpha+\beta)+2}}.
            \end{split}
		\end{align*}
		The sharpness of the constant $D_{0,0,p}$ can be verified by taking
		\[
		g_\rho\left(e^{is}\right):=\left(1-\rho^2\right)^{\frac{2}{1-p}} (1+\cos s)^{\frac{1}{p-1}} e^{-is}.
		\]
	\end{proof}

In what follows, we generalize the results of Theorems \ref{Integral mean thm} and \ref{bound_of_w_z} to all higher-order partial derivatives $\partial^{k} \bar{\partial}^{\,l} w(z)$ for non-negative integers $k,l$. 
	
	\begin{theorem}\label{t3}
		Suppose that $\alpha, \beta\in \mathbb{C}\setminus\mathbb{Z}^-,\, \mbox{\rm Re}(\alpha+\beta)>-1$. Let $w(z)=K_{\alpha,\beta}[f](z)$ be an $(\alpha,\beta)$-harmonic function defined on $\mathbb{D}$ satisfying \eqref{Dirichlet_problem} with $f\in L^p(\mathbb{T}),\, 1\leq p< \infty$, then for $z=re^{i\theta}\in \mathbb{D}$, the following  inequality holds:		
		\begin{align*}
        \begin{split}
			&M_p\left(r, \partial^{k} \bar{\partial}^{\,l} w(z)\right)\\\leq&\frac{|c_{\alpha,\beta}|C_{\alpha,\beta,k,l}}{(1-|z|^2)^{k+l}}\exp\left(\frac{\pi}{2}| \text{\rm Im}(\alpha-\beta)|\right) F\left(-\frac{\mbox{\rm Re}(\alpha+\beta)}{2},-\frac{\mbox{\rm Re}(\alpha+\beta)}{2};1;r^2\right)\|f\|_{L^p(\mathbb{T})},
            \end{split}
		\end{align*}
        where $C_{\alpha,\beta,k,l}$ is defined as in Lemma \ref{Bound_of_u_alpha_beta}.
	\end{theorem}
	\begin{proof}
		Let  $w(z)=K_{\alpha,\beta}[f](z)$ be an $(\alpha,\beta)$-harmonic function defined on  $\mathbb{D}$. 
		It follows from Lemma \ref{Bound_of_u_alpha_beta} and \eqref{poisson_integral} that
		\begin{align*}
        \begin{split}
			&\left| \partial^{k} \bar{\partial}^{\,l} w(z) \right|
			\\=&
			\left|
			\frac{c_{\alpha,\beta}}{2\pi}
			\int_{0}^{2\pi}
			\partial^{k} \bar{\partial}^{\,l} u_{\alpha,\beta}(ze^{-it})\,
			f\left(e^{it}\right)dt
			\right|\\\leq&\frac{|c_{\alpha,\beta}|}{2\pi}
			\int_{0}^{2\pi}
			\abs{\partial^{k} \bar{\partial}^{\,l} u_{\alpha,\beta}(ze^{-it})}\,
			\abs{f\left(e^{it}\right)}dt\\\leq&\frac{|c_{\alpha,\beta}|}{2\pi}
			\int_{0}^{2\pi}C_{\alpha,\beta,k,l}\,
			\frac{|u_{\alpha,\beta}(ze^{-it})|}{(1-|z|^2)^{k+l}}dt\\\leq & \frac{|c_{\alpha,\beta}|}{2\pi}
			\frac{C_{\alpha,\beta,k,l}}{(1-|z|^2)^{k+l}}\int_{0}^{2\pi}\exp\left(\frac{\pi}{2}|\text{\rm Im}(\alpha-\beta)|\right)\frac{(1-|z|^2)^{\text{Re}(\alpha+\beta)+1}}{|1-z|^{\text{Re}(\alpha+\beta)+2}}dt.
            \end{split}
		\end{align*}
		Let $$I_2:=\frac{1}{2\pi}\int_0^{2\pi}\frac{(1-r^2)^{\text{Re}(\alpha+\beta)+1}}{|1-re^{i(\theta-t)}|^{\text{Re}(\alpha+\beta)+2}} d\theta.$$
		By Theorem \ref{Integral mean thm}, we obtain
		$$I_2=F\left(-\frac{\mbox{Re}(\alpha+\beta)}{2},-\frac{\mbox{Re}(\alpha+\beta)}{2};1;r^2\right).$$
		For $1\leq p<\infty$, by Jensen's inequality, we have
		\begin{align*}
        \begin{split}
			\left| \partial^{k} \bar{\partial}^{\,l} w(z) \right|^p \leq &
			 \left(\frac{|c_{\alpha,\beta}|C_{\alpha,\beta,k,l}}{2\pi}\right)^p \frac{\exp\left(\frac{p\pi}{2}| \text{\rm Im}(\alpha-\beta)|\right)I_2^{p-1}}{(1-|z|^2)^{p(k+l)}}\\&\quad\times\int_{0}^{2\pi}\frac{(1-r^2)^{\text{Re}(\alpha+\beta)+1}}{|1-re^{i(\theta-t)}|^{\text{Re}(\alpha+\beta)+2}}|f(e^{it})|^p dt.
            \end{split}
		\end{align*}
		Integrating on both sides of the above inequaltiy and using Fubini's theorem, we get
		\begin{eqnarray*}
			\frac{1}{2\pi} \int_0^{2\pi}\left| \partial^{k} \bar{\partial}^{\,l} w(z) \right|^p d\theta\leq \left(\frac{|c_{\alpha,\beta}|C_{\alpha,\beta,k,l}}{(1-|z|^2)^{k+l}}\right)^p \left(\exp\left(\frac{\pi}{2}|\text{\rm Im}(\alpha-\beta)|\right)\right)^pI_2^p\|f\|^p_{L^p(\mathbb{T})}.
		\end{eqnarray*}
		Therefore, we have
		\begin{align*}
        \begin{split}
			&M_p\left(r, \partial^{k} \bar{\partial}^{\,l} w(z)\right)\\\leq& \frac{|c_{\alpha,\beta}|C_{\alpha,\beta,k,l}}{(1-|z|^2)^{k+l}}\exp\left(\frac{\pi}{2}|\text{\rm Im}(\alpha-\beta)|\right)I_2\cdot\|f\|_{L^p(\mathbb{T})}\\=&\frac{|c_{\alpha,\beta}|C_{\alpha,\beta,k,l}}{(1-|z|^2)^{k+l}}\exp\left(\frac{\pi}{2}|\text{\rm Im}(\alpha-\beta)|\right)
            \\&\quad\times F\left(-\frac{\mbox{Re}(\alpha+\beta)}{2},-\frac{\mbox{Re}(\alpha+\beta)}{2};1;r^2\right)\|f\|_{L^p(\mathbb{T})}\\\leq& \frac{|c_{\alpha,\beta}|C_{\alpha,\beta,k,l}}{(1-|z|^2)^{k+l}}\exp\left(\frac{\pi}{2}|\text{\rm Im}(\alpha-\beta)|\right)B(\alpha,\beta)\|f\|_{L^p(\mathbb{T})}.
            \end{split}
		\end{align*}
    The proof of Theorem \ref{t3} is thus completed.
	\end{proof}
    \begin{rem}
Lemma \ref{Bound_of_u_alpha_beta} gives a pointwise modulus bound for the higher-order partial derivatives of the canonical $(\alpha, \beta)$-harmonic function $u_{\alpha, \beta}(z)$ with a constant depending only on $\alpha, \beta, k, l$, while Theorem \ref{t3} further derives the global $L^p$ integral mean estimates for such derivatives of $(\alpha, \beta)$-harmonic functions.
    \end{rem}

	\section{Properties of $(\alpha,\beta)$-harmonic functions}

    In this section, we apply the integral mean and derivative estimates established in Section 3 to derive additional key properties of $(\alpha,\beta)$-harmonic functions. We first recall the series expansion of $(\alpha,\beta)$-harmonic functions due to Klintberg and Olofsson \cite{KlintbergOlofsson}, then use this expansion to prove that a natural differential operator preserves the $(\alpha,\beta)$-harmonic class. We then establish the subharmonicity of $(\alpha,\beta)$-harmonic functions for positive real parameters $\alpha,\beta$. 
    
	
	\begin{thm} \label{thmB}
		Let $\alpha,\beta\in\mathbb{C}$. Then $u$ is an $(\alpha,\beta)$-harmonic function if and only if
		it has the form
		\begin{equation}\label{series of (alpha,beta)}
			\begin{aligned}
				u(z)
				&= \sum_{m=0}^{\infty}
				c_m\,F(-\alpha,m-\beta; m+1; |z|^2)\,z^m  \\
				&\quad\ \ +\sum_{m=1}^{\infty}
				c_{-m}\,F(-\beta,m-\alpha; m+1; |z|^2)\,\bar z^{\,m}
			\end{aligned}
			\end{equation}
		for some sequence $\{c_m\}_{m=-\infty}^{\infty}$ of complex numbers such that
		\begin{equation}\label{convergence of (alpha,beta)}
			\limsup_{|m|\to\infty} |c_m|^{1/|m|} \leq 1.
		\end{equation}
		Here $F$ denotes the hypergeometric function.
		Moreover, the sums in \eqref{series of (alpha,beta)} are absolutely convergent in the space
		$C^{\infty}(\mathbb{D})$ whenever \eqref{convergence of (alpha,beta)} holds.
	\end{thm}
   
	The linear operator $\mathcal{D}$ acting on harmonic functions is defined by
		\[
		\mathcal{D}
		:= z\frac{\partial}{\partial z}
		- \bar z \frac{\partial}{\partial \bar z},
		\] which plays an important role in fully starlike and convex mappings in $\D$.
        Now, we give some characterizations of $(\alpha,\beta)$-harmonic functions associated with the operator $\mathcal{D}$.
	
	
	\begin{prop}
    The operator $\mathcal{D}$ preserves $(\alpha,\beta)$-harmonic function.
	\end{prop}
	\begin{proof}
		Let $u$ be an $(\alpha,\beta)$-harmonic function with the series expansion. Then 
		
		\begin{align*}
        \begin{split}
		u_z&=\sum_{m=0}^{\infty}c_m\left[F'(-\alpha,m-\beta; m+1; |z|^2)\overline{z}z^m+F(-\alpha,m-\beta; m+1; |z|^2)mz^{m-1}\right]
		 \\& \quad +\sum_{m=1}^{\infty}
		c_{-m}\,F'(-\beta,m-\alpha; m+1; |z|^2)\bar z^{m+1}
        \end{split}
		\end{align*}  and 
		\begin{align*}
        \begin{split}
			u_{\overline{z}}&=\sum_{m=0}^{\infty}
			c_mF'\left(-\alpha,m-\beta; m+1; |z|^2\right)z^{m+1}\\& \quad+\sum_{m=1}^{\infty}c_{-m}\left[F'\left(-\beta,m-\alpha; m+1; |z|^2\right)z\overline{z}^m+F\left(-\beta,m-\alpha; m+1; |z|^2\right)m\overline{z}^{m-1}\right].
		\end{split}
		\end{align*}
		By direct computation, we have 
		\begin{align*}
        \begin{split}
			\mathcal{D}u&=zu_z-\overline{z}u_{\overline{z}}\\&=\sum_{m=0}^{\infty}c_mF(-\alpha,m-\beta; m+1; |z|^2)mz^{m}-\sum_{m=1}^{\infty}c_{-m}F(-\beta,m-\alpha; m+1; |z|^2)m\overline{z}^{m}.
		\end{split}
		\end{align*}
		Furthermore, for the sequence $\{c_m\}_{m=-\infty}^\infty$, if \eqref{convergence of (alpha,beta)} holds, then 
		$$\limsup\limits_{\abs{m}\rightarrow\infty}|mc_m|^{\frac{1}{|m|}}\leq 1$$ and $$\limsup\limits_{\abs{m}\rightarrow\infty}|-mc_{-m}|^{\frac{1}{|m|}}\leq 1.$$ Therefore, by Theorem \ref{thmB}, we deduce that $\mathcal{D}u$ is an $(\alpha,\beta)$-harmonic function.
	\end{proof}
	\begin{prop} For
    $\alpha>0$ and $\beta>0$, $(\alpha,\beta)$-harmonic functions are subharmonic on $\mathbb{D}$.
\end{prop}
\begin{proof}
	By Theorem \ref{thmB}, we know that $F:=F\left(-\alpha,-\beta;1;|z|^2\right)$ is an $(\alpha,\beta)$-harmonic  function. Note that $$\Delta_{\alpha,\beta} F = 0.$$ It follows that
	\begin{equation}\label{expression for Del}
		\frac{1}{4}\left(1-|z|^2\right)\Delta F=\alpha\beta F-\alpha z\frac{\partial F}{\partial z}-\beta\overline{z}\frac{\partial  F}{\partial \overline{z}}.
	\end{equation} 
	Let $u=|z|^2$. Then $F_z=\overline{z}F_u$ and $F_{\overline{z}}=zF_u.$ Now, we have
	\begin{align}
    \begin{split}\label{simplified Del}
		\alpha\beta F-\alpha z\frac{\partial F}{\partial z}-\beta\overline{z}\frac{\partial  F}{\partial \overline{z}}=\alpha\beta F-\alpha|z|^2F_u-\beta|z|^2F_u=\alpha\beta F\cdot\left(1-\frac{\alpha+\beta}{\alpha\beta}\frac{uF_u}{F}\right).
        \end{split}
	\end{align} 
	If $$F=F(-\alpha,-\beta;1;u):=\sum_{n=0}^\infty a_nu^n,$$ then $$F_u=\sum_{n=1}^\infty na_nu^{n-1}$$ and $$uF_u=\sum_{n=1}^\infty na_nu^n:=\sum_{n=1}^\infty b_nu^n.$$ By Lemma \ref{seriesconverge}, we see that $b_n/a_n=n$ is increasing, which implies that $\frac{uF_u}{F}$ is strictly increasing on $(0,1)$. It follows from Lemma \ref{Lemma1}, \eqref{F(a,b;c;x)prop_1} and \eqref{F(a,b;c;x)prop_2} that
	\begin{equation*}
		\frac{uF_u}{F}<\lim\limits_{u\rightarrow 1}\frac{uF_u}{F}=\frac{\alpha\beta F(-\alpha+1,-\beta+1;2;1)}{F(-\alpha,-\beta;1;1)}=\frac{\alpha\beta}{\alpha+\beta}.
	\end{equation*}
Note that  $F\left(-\alpha,-\beta;1;|z|^2\right)>0$, $\alpha>0$ and $\beta>0$. By \eqref{expression for Del} and \eqref{simplified Del}, we get
	\begin{eqnarray*}
		\frac{1}{4}\left(1-|z|^2\right)\Delta F\geq 0.
	\end{eqnarray*}
	This shows that $F\left(-\alpha,-\beta;1;|z|^2\right)$ is subharmonic on $\ID$.
\end{proof}

\begin{theorem}\label{t4}
	Let $w(z)=K_{\alpha,\beta}[f](z)$ be an $(\alpha,\beta)$-harmonic function with $\mbox{\rm Re}(\alpha+\beta)>-1$ defined on  $\mathbb{D}$  with $f\in L^p(\mathbb{T}),1\leq p\leq \infty$. If $w(z)$ is of the form \eqref{series of (alpha,beta)}, then 
	\begin{itemize}
		\item [(i)] $|c_k|\leq |c_{\alpha,\beta}|\dfrac{(\alpha+1)_{k}}{k!}\|f\|_{L^p(\mathbb{T})},\ \ |c_{-k}|\leq|c_{\alpha,\beta}|\dfrac{(\beta+1)_{k}}{k!}\|f\|_{L^p(\mathbb{T})}.$
		\item [(ii)]$\dfrac{k!}{(\beta+1)_{k}}|c_k|+\dfrac{k!}{(\alpha+1)_{k}}|c_{-k}|\leq 2|c_{\alpha,\beta}|C_q\|f\|_{L^p(\mathbb{T})}$.
	\end{itemize}
In particular, when $p=\infty$, $C_q=\frac{2}{\pi}$ and $\alpha=\beta=0$, we have $$|c_k|+|c_{-k}|\leq\frac{4}{\pi}\|f\|_{\infty}.$$
\end{theorem}
\begin{proof}
	Let $w(z)$ be an $(\alpha,\beta)$-harmonic function defined on $\ID$. By \eqref{poisson_integral},
    we have
	\begin{equation}\label{47}
			w(z) = \frac{1}{2\pi} \int_0^{2\pi} c_{\alpha,\beta} \frac{\left(1 - |z|^2\right)^{\alpha + \beta + 1}}{(1 - z e^{-it})^{\alpha+1} (1 - \overline{z} e^{it})^{\beta+1}} f\left(e^{it}\right)dt.
	\end{equation}
	Now, by direct calculation, we get
	
	\begin{align*}
		w_{z}(0)
		&= \frac{c_{\alpha,\beta}}{2\pi}\int_{0}^{2\pi} (\alpha+1)\,e^{it} f(e^{it})\,dt,\\
		w_{ z}^{(2)}(0)
		&= \frac{c_{\alpha,\beta}}{2\pi}\int_{0}^{2\pi} (\alpha+2)(\alpha+1)\,e^{2it} f(e^{it})\,dt,\\ 
		&\ \vdots \notag\\
		w_{ z}^{(k)}(0)
		&= \frac{c_{\alpha,\beta}}{2\pi}\int_{0}^{2\pi} (\alpha+1)_{k}\, e^{ikt} f(e^{it})\,dt.
		\end{align*}
	Similarly, we have
	$$w_{ \overline{z}}^{(k)}(0)
	= \frac{c_{\alpha,\beta}}{2\pi}\int_{0}^{2\pi} (\beta+1)_{k}\, e^{ikt} f(e^{it})\,dt.
	$$
Thus, for \(k=1,2,\ldots\), we know that 
	$$
	c_k
	= \frac{w_z^{(k)}(0)}{k!}
	= \,
	\frac{c_{\alpha,\beta}}{2\pi}\int_{0}^{2\pi} \frac{(\alpha+1)_{k}}{k!}\,e^{ikt} f\left(e^{it}\right)dt.
	$$
	Moreover, we find that
	$$
	c_{-k}
	= \frac{w_{\bar z}^{(k)}(0)}{k!}
	= \,
	\frac{1}{2\pi}\int_{0}^{2\pi} \frac{(\beta+1)_{k}}{k!}\,e^{ikt} f\left(e^{it}\right)\,dt.
	$$
	Applying H\"older's inequality, we get
	$$|c_k|\leq |c_{\alpha,\beta}|\frac{(\alpha+1)_{k}}{k!}\|f\|_{L^p(\mathbb{T})}\ \ \text{and}\ \ |c_{-k}|\leq|c_{\alpha,\beta}|\frac{(\beta+1)_{k}}{k!}\|f\|_{L^p(\mathbb{T})}.$$
Let $c_k=|c_k|e^{i\alpha_k}$, $c_{-k}=|c_{-k}|e^{-i\beta_k}$ and $\theta_k=\dfrac{\alpha_k+\beta_k}{2}$. Then 
$$
|c_k|=\frac{(\alpha+1)_{k}}{ k!}\frac{c_{\alpha,\beta}}{2\pi}
\int_{0}^{2\pi}
f\left(e^{it}\right)e^{-ikt}e^{-i\alpha_k}dt$$ and
$$
|c_{-k}|=\frac{(\beta+1)_{k}}{ k!}
\frac{c_{\alpha,\beta}}{2\pi}\int_{0}^{2\pi}
f\left(e^{it}\right)e^{ikt}e^{i\beta_k}dt.
$$
Therefore, we deduce that
\begin{align*}
\begin{split}
	&\frac{ k!}{(\alpha+1)_{k}}\,|c_k|+ \frac{k!}{(\beta+1)_{k}}\,|c_{-k}|
	\\=&\left|
	\frac{c_{\alpha,\beta}}{2\pi i}\int_{0}^{2\pi}f\left(e^{it}\right) \left(e^{-ikt}e^{-i\alpha_k}+e^{ikt}e^{i\beta_k}\right)dt\right|\\ \leq&\frac{|c_{\alpha,\beta}|}{2\pi}
	\int_{0}^{2\pi}\abs{f\left(e^{it}\right)}\,\bigl|e^{-ikt}e^{-i\alpha_k}+e^{ikt}e^{i\beta_k}\bigr|\,dt
	\\=&\frac{|c_{\alpha,\beta}|}{\pi}\int_{0}^{2\pi}\abs{f\left(e^{it}\right)}\,\bigl|\cos k(t+\theta_k)\bigr|dt\\ \leq&
	2|c_{\alpha,\beta}|\left(\frac{1}{2\pi}\int_{0}^{2\pi}
	\abs{f\left(e^{it}\right)}^{p}\,dt\right)^{\frac{1}{p}}
	\left(\frac{1}{2\pi}\int_{0}^{2\pi}
	|\cos k(t+\theta_k)|^{q}\,dt\right)^{\frac{1}{q}}\\=&
	2|c_{\alpha,\beta}|C_q\,\|f\|_{L^p(\mathbb{T})} ,
    \end{split}
	\end{align*}
where
\[
C_q
=
\left(
\frac{1}{2\pi}
\int_{0}^{2\pi}
|\cos(kt)|^{q}\,dt
\right)^{\frac{1}{q}},
\]
and \(q\) is the conjugate exponent of \(p\), satisfying
\(\frac{1}{p}+\frac{1}{q}=1\).
\end{proof}

By setting $\alpha=\beta=0$ in Theorem \ref{t4}, we get the following result due to Shi, Li and Lian 
\cite[Theorem 2.1]{sll}.

\begin{cor}
Suppose that $F \in L^{p}(\mathbb{T})$ for $1 \le p \le \infty$ and
$f = \mathcal{P}[F]$ is a harmonic mapping in $\mathbb{D}$.
If
\[
f(z) = \sum_{n=0}^{\infty} a_n z^n + \sum_{n=1}^{\infty} \overline{b_n}\,\overline{z}^n
\]
for any positive integer $n$, then 
\begin{itemize}
\item[(i)]
		
$|a_n|,\ |b_n| \le \|f\|_{L^p(\mathbb{T})} .$ The equalities hold for the functions $F\left(e^{it}\right) = e^{int}$ and
$F\left(e^{it}\right) = e^{-int}$, respectively.
		
\item[(ii)]
$
|a_n| + |b_n| \le 2 C_q \|f\|_{L^p(\mathbb{T})} ,
$
where
\[
C_q = \left( \frac{1}{2\pi} \int_0^{2\pi} |\cos(nt)|^q \, dt \right)^{1/q} \le 1
\ \ \text{and} \ \ \frac{1}{p} + \frac{1}{q} = 1.
	\]
\end{itemize}
	Especially, when $p = \infty$, $C_q = \frac{2}{\pi}$, we have
\[
|a_n| + |b_n| \le \frac{4}{\pi} \|F\|_\infty.
\]
\end{cor}

\begin{theorem}\label{t5}
	Let $w(z)=K_{\alpha,\beta}[f](z)$ be an $(\alpha,\beta)$-harmonic function with ${\rm Re}(\alpha+\beta)>-1$ defined on  $\mathbb{D}$  with $f\in L^p(\mathbb{T}),1\leq p\leq \infty$. Then for all  $z=re^{i\theta}$, 
	\begin{eqnarray*}
			&&\frac{1}{\left(|\alpha+1|+|\beta||z|\right)}|w_z|+\frac{1}{\left(|\beta+1|+|\alpha||z|\right)}|w_{\overline{z}}|\\&& \quad\leq\begin{cases}
				2|c_{\alpha,\beta}|\exp\left(\frac{\pi}{2}|\mbox{\rm Im}(\alpha- \beta)|\right)\dfrac{(1-|z|^2)^{\mbox{\rm Re}(\alpha+\beta)}}{(1-|z|)^{\mbox{\rm Re}(\alpha+\beta)+2}}\|f\|_{L^1(\mathbb{T})} &  (p=1),\\\dfrac{2|c_{\alpha,\beta}|\|f\|_{L^p(\mathbb{T})}}{(1-|z|^2)^{1+1/p}}\exp\left(\frac{\pi}{2}|\mbox{\rm Im}(\alpha- \beta)|\right)F_2&(1<p\leq \infty).
			\end{cases}
		\end{eqnarray*}
		where \begin{equation}\label{00}F_2=F\left(1-\frac{(\text{\rm Re}(\alpha+\beta)+2)q}{2},1-\frac{(\text{\rm Re}(\alpha+\beta)+2)q}{2};1;r^2\right)^{\frac{1}{q}}.\end{equation}
\end{theorem}
\begin{proof}
	By Lemma \ref{partial derivative of K}, we have 
   \begin{align*}
   \begin{split}
       &w_z(z)\\=&\frac{1}{2\pi} \int_0^{2\pi} c_{\alpha,\beta} \bigg(\frac{(\alpha + 1)(1-|z|^2)^{\alpha+\beta}}{(1-ze^{-it})^{\alpha+2}(1-\overline{z}e^{it})^\beta} - \frac{\beta z (1-|z|^2)^{\alpha+\beta}}{(1-ze^{-it})^{\alpha+1}(1-\overline{z}e^{it})^{\beta+1}}\bigg) f\left(e^{it}\right) dt
       \end{split}
       \end{align*} 
       and \begin{align*}
   \begin{split}&w_{\overline{z}}(z)\\=&\frac{1}{2\pi} \int_0^{2\pi} c_{\alpha,\beta} \bigg(\frac{(\beta + 1)(1-|z|^2)^{\alpha+\beta}}{(1-ze^{-it})^{\alpha}(1-\overline{z}e^{it})^{\beta+2}} - \frac{\alpha z (1-|z|^2)^{\alpha+\beta}}{(1-ze^{-it})^{\alpha+1}(1-\overline{z}e^{it})^{\beta+1}}\bigg) f\left(e^{it}\right) dt.\end{split}
       \end{align*} By \eqref{WVB_3.3}, we have
	\begin{align}
    \begin{split}
    \label{bound_|w_z|+|w_bar{z}|}
		&\frac{1}{|\alpha+1|+|\beta||z|}|w_z|+\frac{1}{|\beta+1|+|\alpha||z|}|w_{\overline{z}}|\\& \quad\leq \frac{2|c_{\alpha,\beta}|}{\pi}\int_0^{2\pi} \exp\left(\frac{\pi}{2}|\mbox{\rm Im}(\alpha- \beta)|\right)\dfrac{(1-|z|^2)^{\mbox{\rm Re}(\alpha+\beta)}}{(|1-ze^{-it}|)^{\mbox{\rm Re}(\alpha+\beta)+2}}\abs{f\left(e^{it}\right)}dt.
	\end{split}
    \end{align}
	For $p=1$, we obtain
	\begin{align}
    \begin{split}
		&\frac{1}{|\alpha+1|+|\beta||z|}|w_z|+\frac{1}{|\beta+1|+|\alpha||z|}|w_{\overline{z}}|\\& \quad\leq  2|c_{\alpha,\beta}|\exp\left(\frac{\pi}{2}|\mbox{\rm Im}(\alpha- \beta)|\right)\dfrac{(1-|z|^2)^{\mbox{\rm Re}(\alpha+\beta)}}{(1-|z|)^{\mbox{\rm Re}(\alpha+\beta)+2}}\|f\|_{L^1(\mathbb{T})}.
	\end{split}
    \end{align}
	For $1<p\leq \infty$, combining
H\"older's inequality, \eqref{bound_|w_z|+|w_bar{z}|} with Theorem \ref{bound_of_w_z}, we have
	\begin{align*}
    \begin{split}
		&\frac{1}{|\alpha+1|+|\beta||z|}|w_z|+\frac{1}{|\beta+1|+|\alpha||z|}|w_{\overline{z}}|\\\leq&2|c_{\alpha,\beta}|\left(\frac{1}{2\pi}\int_{0}^{2\pi}
		|f(e^{it})|^{p}\,dt\right)^{\frac{1}{p}} \\&\quad \times\left(\frac{1}{2\pi}\int_{0}^{2\pi}\dfrac{\left(\exp\left(\frac{\pi}{2}|\mbox{\rm Im}(\alpha- \beta)|\right)\right)^q\left(1-|z|^2\right)^{\left(\mbox{\rm Re}(\alpha+\beta)\right) q}}{(|1-ze^{-it}|)^{\left(\mbox{\rm Re}(\alpha+\beta)+2\right)q}} dt\right)^{\frac{1}{q}}\\=&\dfrac{2|c_{\alpha,\beta}|\|f\|_{L^p(\mathbb{T})} }{\left(1-|z|^2\right)^{1+1/p}}\exp\left(\frac{\pi}{2}|\mbox{\rm Im}(\alpha- \beta)|\right)F_2,   
    \end{split}
	\end{align*}
where $F_2$ is given by \eqref{00}. This completes the proof of Theorem \ref{t5}.
\end{proof}

\begin{rem}
    By setting $p=\infty$ and $\alpha=\beta=0$ in Theorem \ref{t5}, we get  the classical Schwarz-Pick lemma for bounded harmonic mappings. 
\end{rem}
	\vskip.02in
	\noindent{\bf Acknowledgements.} 
	Z.-G. Wang was partially supported by the \textit{Key Project of Education Department of Hunan Province} under Grant no. 25A0668, and the \textit{Natural Science Foundation of Changsha} under Grant no. kq2502003 of the P. R. China. Brindha Valson E and R. Vijayakumar acknowledge the Department of Science and Technology (DST), Government of India, for supporting the Department of Mathematics at NIT Calicut under the FIST program, which enabled this research. 
    The authors thank Dr. Qingtian Shi for his valuable comments and suggestions during the preparation of this paper.

	

	
		\vskip.05in
	\noindent{\bf Authors' contributions.} All authors contributed equally to this work.
	
		\vskip.05in
	\noindent{\bf Conflicts of interest.} The authors declare that they have no conflict of interest.
	
	\vskip.05in
	\noindent{\bf Data availability.} Data sharing does not apply to this article, as no datasets were generated or analyzed during the current study.
		
\end{document}